\documentclass[reqno,12pt]{amsart}
\usepackage{amsmath}
\usepackage{amssymb}
\usepackage{amsfonts}
\usepackage{mathrsfs}
\usepackage{bbding}
\usepackage{mathtools}
\usepackage{amssymb,amsmath,amscd,graphicx,latexsym}
 \usepackage{xcolor}
\newtheorem{THM}{Theorem}[section]

\newtheorem{LEM}[THM]{Lemma}
\newtheorem{PROP}[THM]{Proposition}
\newtheorem{COR}[THM]{Corollary}
\newtheorem{DEF}[THM]{Definition}
\newcommand{\Def}[2]
{
\begin{DEF}[#1]
\emph{#2}
\end{DEF}
}

\numberwithin{equation}{section}
\newcommand{\mat}[3]
{
\begin{pmatrix}
\begin{smallmatrix}
1 & {#2} & {#3} \\
0 & 1 & {#1} \\
0 & 0 & 1
\end{smallmatrix}
\end{pmatrix}
}
\newcommand{\liemat}[3]
{
\begin{pmatrix}
\begin{smallmatrix}
0 & {#2} & {#3} \\
0 & 0 & {#1} \\
0 & 0 & 0
\end{smallmatrix}
\end{pmatrix}
}
\newcommand{\bmat}[3]
{
\begin{pmatrix}
1 & {#2} & {#3} \\
0 & 1 & {#1} \\
0 & 0 & 1
\end{pmatrix}
}

\newcommand{\Span}{\mathrm{Span}}
\newcommand{\Nil}{\Gamma \backslash G}
\newcommand{\TNil}{\mathbb{T} \times\Gamma \backslash G}
\newcommand{\lb}{\left(}
\newcommand{\rb}{\right)}

\begin{document}
\title[M\"{o}bius disjointness for skew products on a circle and a nilmanifold]
{M\"{o}bius disjointness for skew products \\ on a circle and a nilmanifold}
\author{Wen Huang, Jianya Liu \& Ke Wang}

\address{Wu Wen-Tsun Key Laboratory of Mathematics, USTC, Chinese Academy
of Sciences \& Department of Mathematics
\\ University of Science and Technology
of China
\\
Hefei\\
 Anhui 230026
\\China}
\email{wenh@mail.ustc.edu.cn}

\address{School of Mathematics
\\
Shandong University
\\
Jinan
\\
Shandong 250100
\\
China}
\email{jyliu@sdu.edu.cn}

\address{School of Mathematics
\\
Shandong University
\\
Jinan
\\
Shandong 250100
\\
China}
\email{wangkesdu@gmail.com}

\date{\today}


\begin{abstract}
Let $\mathbb{T}$ be the unit circle and $\Nil$ the $3$-dimensional Heisenberg nilmanifold.
We prove that a class of skew products on $\mathbb{T} \times \Nil$ are distal, and that
the M\"{o}bius function is linearly disjoint from these skew products.
This verifies the M\"{o}bius Disjointness Conjecture of Sarnak.
\end{abstract}

\maketitle
{\small\tableofcontents}

\section{Introduction}

Let $\mu(n)$ be the M\"obius function, that
is $\mu(n)$  is $0$ if $n$ is not square-free, and is $(-1)^k$ if  $n$ is a
product of $k$ distinct primes. The behavior of $\mu$ is central in the theory of prime numbers.
Let $(X, T)$ be a flow, namely $X$ is a compact metric space
and $T: X\to X$ a continuous map. We say that $\mu$ is linearly
disjoint from $(X, T)$ if
\begin{eqnarray}\label{def/DISJO}
\lim_{N\to\infty} \frac{1}{N}\sum_{n\leq N} \mu(n)f(T^n x)=0
\end{eqnarray}
for any $f\in C(X)$ and any $x\in X$.
The M\"{o}bius Disjointness Conjecture of Sarnak \cite{Sar09} \cite{Sar12}
states that the function $\mu$ is linearly disjoint from
every $(X, T)$ whose entropy is $0$.
This conjecture has been proved for many cases, and we refer to the survey paper
\cite{FKL18} for recent progresses.
An incomplete list for works related to the present paper is: Bourgain \cite{Bou13},
Green-Tao \cite{GT12b}, Liu-Sarnak \cite{LS15, LS17}, Wang \cite{Wan17},
Peckner \cite{Pec18}, Huang-Wang-Ye \cite{HWY19}, and Litman-Wang \cite{LW19}.

Distal flows are typical examples of zero-entropy flows; see Parry \cite{Par68}.
A flow $(X,T)$ with a compatible metric $d$ is called {\it distal} if
$$
\inf_{n \geq 0}d(T^n x, T^n y)>0
$$
whenever $x \neq y$. According to Furstenberg's structure theorem of minimal distal flows \cite{Fus63},
skew products are building blocks of distal flows.

An example of distal flow is the skew product $T$ on the $2$-torus
$\mathbb{T}^2 = (\mathbb{R}/\mathbb{Z})^2$ given by
\begin{equation}\label{1.T2SKEWPRODUCT}
T: (x,y) \mapsto (x+\alpha,y+h(x)),
\end{equation}
where $\alpha \in [0,1)$ and $h \colon \mathbb{T} \rightarrow \mathbb{R}$ is a continuous
function. For dynamical properties of this skew product, see for example Furstenberg \cite{Fus61}.
The M\"{o}bius disjointness for the skew product (\ref{1.T2SKEWPRODUCT})
was first studied by Liu and Sarnak in \cite{LS15, LS17}. A result in \cite{LS15} states that,
if $h$ is analytic with an additional assumption on its Fourier coefficients,
then the M\"{o}bius Disjointness Conjecture is true for the skew product $(\mathbb{T}^2, T)$.
This result holds for all $\alpha$, as is not common in the KAM theory.
The aforementioned additional assumption was removed in Wang \cite{Wan17}.
It has been further generalized by Huang, Wang and Ye in \cite{HWY19}
to the case that $h(x)$ is $C^{\infty}$-smooth.

Another example of distal flow is nilsystem. Let $G$ be a nilpotent
Lie group with a discrete cocomapct subgroup $\Gamma$. The group $G$ acts in a natural
way on the homogeneous space $\Gamma \backslash G$. Fix $h \in G$. Then the
transformation $T$ given by $T(\Gamma g)=\Gamma g h$ makes $(\Nil,T)$ a nilsystem.
The M\"{o}bius Disjointness Conjecture for these nilsystems was proved by Green and Tao in \cite{GT12b}.

\medskip

Now let $G$ be the $3$-dimensional Heisenberg group with the cocompact discrete subgroup $\Gamma$, namely
\begin{equation}
G=\mat{\mathbb{R}}{\mathbb{R}}{\mathbb{R}}, \quad
\Gamma=\mat{\mathbb{Z}}{\mathbb{Z}}{\mathbb{Z}}.
\end{equation}
Then $\Nil$ is the $3$-dimensional Heisenberg nilmanifold. Let $\mathbb{T}$ be the unit circle.
This paper is devoted to the M\"{o}bius Disjointness Conjecture for skew products on $\mathbb{T} \times \Nil$,
and the main result is as follows.

\begin{THM}\label{1*.MT}
Let $\mathbb{T}$ be the unit circle and $\Nil$ the $3$-dimensional Heisenberg nilmanifold.
Let $\alpha \in [0,1)$ and let $\varphi, \psi$ be $C^{\infty}$-smooth periodic functions from
$\mathbb{R}$ to $\mathbb{R}$ with period $1$ such that
\begin{equation}\label{1.ASSMFORphi}
\int_{0}^{1} \varphi(t) \, \mathrm{d}t = 0.
\end{equation}
Let the skew product $T$ on $\TNil$ be given by
\begin{equation}\label{1.TForm}
T: (t,\Gamma g) \mapsto \lb t+\alpha, \Gamma g \bmat{\varphi(t)}{\varphi(t)}{\psi(t)}\rb.
\end{equation}
Then, for any $(t_0,\Gamma g_0) \in \TNil$ and any $f \in C(\TNil)$,
\begin{equation}\label{1.Main}
\lim_{N \rightarrow \infty} \frac{1}{N}\sum_{n=1}^N \mu(n)f(T^n(t_0,\Gamma g_0))=0.
\end{equation}
\end{THM}

Note that Theorem \ref{1*.MT} holds for all $\alpha$.
The flow $(\TNil, T)$ in Theorem \ref{1*.MT} is distal, as is implied in
Theorem \ref{Thm/Distal} of the present paper. Thus Theorem \ref{1*.MT} verifies
the M\"{o}bius Disjointness Conjecture in this context.

To prove this theorem, we first construct in Section 2 a subset ${\mathcal F}\subset C(\TNil)$
which spans a $\mathbb{C}$-linear subspace that is dense in $C(\TNil)$, so that the proof of
\eqref{1.Main} is reduced to that for those special $f\in {\mathcal F}$ having explicit forms.
Theorem \ref{1*.MT} for rational $\alpha$ depends on known results for skew products on $\mathbb{T}^2$
as well as a classical theorem of Hua \cite{Hua65} on exponential
sums over primes in arithmetic progressions; this is done in Section 3. The case of irrational $\alpha$ occupies
the bulk of the paper. Sections 4 and 5 are preparatory: in Section 4, we study the rational
approximations of $\alpha$ and their
analytic consequences; in Section 5 we introduce the concept,
as well as the computation, of measure complexity.  The whole Section 6 is devoted to the proof that, for irrational $\alpha$,
the measure complexity of $(\TNil, T, \rho)$, where $\rho$ is any $T$-invariant Borel probability
measure on $\TNil$, is sub-polynomial. Theorem \ref{1*.MT} for irrational $\alpha$
finally follows from this and the main theorem of Huang-Wang-Ye \cite{HWY19}.
We remark that the work of Matom\"{a}ki-Radziwill-Tao \cite{MRT15} makes it possible to use measure
complexity, rather than the original topological entropy, to investigate the M\"{o}bius disjointness.

We conclude this introduction by reporting some thoughts about generalizations.
A careful reader would naturally ask whether the M\"{o}bius disjointness could be established for
more general skew products $S$ of the form, instead of \eqref{1.TForm},
\begin{equation}\label{1/Gen/S}
S: (t,\Gamma g) \mapsto \lb t+\alpha, \Gamma g \bmat{\varphi_1(t)}{\varphi_2(t)}{\psi(t)}\rb
\end{equation}
where $\varphi_1, \varphi_2, \psi$ are three $C^{\infty}$-smooth periodic functions with period $1$.
The $S$ in \eqref{1/Gen/S} is more general in the sense that the two functions
$\varphi_1$ and $\varphi_2$ are not necessarily the same. We prove in Theorem \ref{Thm/Distal} that
the flow $(\TNil, S)$ is distal, and hence has zero entropy, for which the M\"{o}bius Disjointness Conjecture
is expected to hold. Our method in this paper, however, does not directly apply to $(\TNil, S)$, and the reason
is pointed out in the footnote to \eqref{6/SimPi} in Section 6. It seems an interesting question
to establish the M\"{o}bius Disjointness Conjecture for $(\TNil, S)$. Another question is to
study M\"{o}bius disjointness for general skew products on $\TNil$ where $\Nil$ are high
dimensional nilmanifolds. We plan to come back to these in other occasions.

While finishing this manuscript, we notice that Kanigowski, Lemanczyk and Radziwill \cite{KLR19} has
studied the M\"{o}bius disjointness for skew products \eqref{1.T2SKEWPRODUCT} on ${\mathbb T}^2$
where $h$ is absolutely continuous. It is therefore possible to generalize our Theorem \ref{1*.MT} to the case of
absolutely continuous $\varphi$ and $\psi$ by similar arguments.

\medskip

\noindent
\textbf{Notations.} We list some notations that we use in the paper. We write $e(x)$
for $e^{2\pi i x}$, and write $\|x\|$ for the distance between $x$ and the nearest integer,
that is
$$
\|x\|=\min_{n \in \mathbb{Z}} |x-n|.
$$
For positive $A$, the notations
$B=O(A)$ or $B \ll A$ mean that there exists a positive constant $c$ such that $|B| \leq cA$.
If the constant $c$ depends on a parameter $b$, we write
$B=O_{b}(A)$ or $B \ll_{b} A$. The notation $A \asymp B$ means that $A \ll B$ and $B \ll A$.
For a topological space $X$, we use $C(X)$ to denote the set of all continuous
complex-valued functions on $X$. If $X$ is a smooth
manifold and $r \geq 1$ is an integer, then we use $C^r(X)$ to denote the set of all $f \in C(X)$ that have continuous $r$-th derivatives.

\section{Approximations on $C(\TNil)$}
Let $G$ be the $3$-dimensional Heisenberg group with the cocompact discrete subgroup $\Gamma$,
and $\Nil$ the $3$-dimensional Heisenberg nilmanifold. The purpose of this section is to
construct a subset of $C(\TNil)$,
which spans a $\mathbb{C}$-linear subspace that is dense in $C(\TNil)$. A basic
reference for this section is Tolimieri \cite{Tol77}.

For integers $m,j$ with $0 \leq j \leq m-1$, define the functions $\psi_{mj}$ and $\psi_{mj}^{*}$ on $G$ by
\begin{equation*}
\psi_{mj}\mat{x}{y}{z}=e(mz+jx) \sum_{k \in \mathbb{Z}}e^{-\pi(y+k+\frac{j}{m})^2}e(mkx),
\end{equation*}
and
\begin{equation*}
\psi_{mj}^* \mat{x}{y}{z}= ie(mz+jx) \sum_{k \in \mathbb{Z}} e^{-\pi(y+k+\frac{j}{m}+\frac{1}{2})^2}e\lb\frac{1}{2}\lb y+k+\frac{j}{m}\rb +mkx\rb.
\end{equation*}
We check that $\psi_{mj}$ and $\psi_{mj} ^*$ are $\Gamma$-invariant, that is
\begin{equation*}
\psi_{mj}(\gamma g)=\psi_{mj}(g), \quad\psi_{mj}^*(\gamma g)=\psi_{mj}^*(g)
\end{equation*}
for any $g \in G$ and for any $\gamma \in \Gamma.$
Thus $\psi_{mj}$ and $\psi_{mj}^*$ can be regarded as functions on the nilmanifold $\Nil$.

Recall that there is a unique Borel probability measure on $\Nil$ that
is invariant under the right translations, and therefore $L^2(\Nil)$ can be defined.
For $m \in \mathbb{Z}$ let $V_m$ be the subspace of $L^2(\Nil)$ consisting
of all functions $f\in L^2(\Nil)$ satisfying
\begin{equation*}
f\lb \Gamma g \mat{0}{0}{z} \rb= e(mz)f \lb \Gamma g\rb
\end{equation*}
for any $g \in G$ and for any $z \in \mathbb{R}$.
Then the space $L^2(\Nil)$ decomposes into the following orthogonal direct sum:
\begin{equation*}
L^2(\Nil)=\bigoplus_{m \in \mathbb{Z}} V_m.
\end{equation*}
Set $C_m(\Nil)=V_m \cap C(\Nil)$ and $C_{m}^{r} = V_m \cap C^r(\Nil)$.
Let $p_m$ be the projection of $L^2(\Nil)$ onto $V_m$. Then we may write $p_m$ explicitly in the form
\begin{equation*}
p_m(f) \lb \Gamma \mat{x}{y}{z} \rb =\int_{0}^{1} f\lb \Gamma \mat{x}{y}{z+t} \rb e(-mt) \, \mathrm{d}t.
\end{equation*}
Hence we have $p_m(f) \in C_m^r (\Nil)$ if $f \in C^r(\Nil)$.

We need the following two results from \cite{Tol77};
the first is \cite[Lemma 6.3]{Tol77}, and the second is \cite[Lemma 6.2]{Tol77}. Note that the original \cite[Lemma 6.2]{Tol77}
is slightly stronger than the one we state here.

\begin{LEM}\label{2*.Tol1}
Let $F \in C^r(\Nil)$ with $r \geq 1$,
and write $F=\sum_{m \in \mathbb{Z}} F_m$ with $F_m \in V_m$. Then the series $\sum_{m \in \mathbb{Z}} F_m$ is absolutely and uniformly convergent.
\end{LEM}

\begin{LEM}\label{2*.Tol2}
Let $F \in C_{m}^{r}(\Nil)$ with $m \geq 1$ and $r \geq 1$. Then there exist functions
$h_j, h_j^{*} \in C^{r-1}(\mathbb{T}^2), j=0,1, \ldots, m-1$, such that
\begin{equation*}
F\lb \Gamma \mat{x}{y}{z}\rb = \sum_{j=0}^{m-1} (h_j(x,y)\psi_{mj}(x,y,z)+h_{j}^{*}(x,y)\psi_{mj}^{*}(x,y,z)).
\end{equation*}
\end{LEM}
Now we construct the desired subset of $C(\TNil)$.
\begin{PROP}\label{2*.MP}
Let $\mathcal{A}$ be the subset of $f \in C(\TNil)$ such that
\begin{equation*}
f: \lb t,\Gamma \mat{x}{y}{z} \rb \mapsto e(\xi_1 t + \xi_2 x + \xi_3 y) \psi\lb \Gamma \mat{x}{y}{z}\rb
\end{equation*}
where $\xi_1,\xi_2,\xi_3 \in \mathbb{Z}$, and $\psi = \psi_{mj}, \overline{\psi}_{mj}, \psi_{mj}^{*}$ or $\overline{\psi}^{*}_{mj}$
for some $0 \leq j \leq m-1$. Here $\overline{\psi}_{mj}$ and $\overline{\psi}^{*}_{mj}$ stand for the complex
conjugates of $\psi_{mj}$ and $\psi_{mj}^{*}$, respectively. Let $\mathcal{B}$ be subset of $f \in C(\TNil)$
satisfying
\begin{equation*}
f: (t,\Gamma g) \mapsto f_1(t)f_2(\Gamma g)
\end{equation*}
with $ f_1\in C(\mathbb{T})$ and $f_2 \in C_0(\Nil)$.
Then the $\mathbb{C}$-linear subspace spanned by $\mathcal{A} \cup \mathcal{B}$ is dense in $C(\TNil)$.
\end{PROP}

\begin{proof}
By the Stone-Weierstrass theorem, the $\mathbb{C}$-linear subspace of $C(\TNil)$ spanned by
\begin{equation*}
{\mathcal F} =\{f:  f(t,\Gamma g) = e(\xi_1 t) F(\Gamma g),  \xi_1 \in \mathbb{Z}, F \in C(\Nil) \}
\end{equation*}
is dense. Thus it suffices to investigate the approximations on $C(\Nil)$. Since $C^1(\Nil)$ is dense in $C(\Nil)$, we can consider $C^1(\Nil)$ instead. By Lemma~\ref{2*.Tol1}, any $F \in C^1(\Nil)$ can be written as
$$
F=\sum_{m \in \mathbb{Z}}F_m
$$
with $F_m \in V_m$, and this series is absolutely and uniformly convergent. Hence $F$ can be approximated arbitrarily close
by the sum of finitely many $F_m$. Therefore, we need only to investigate each $F_m$ with $m \in \mathbb{Z}$. Since $F \in C^1(\Nil)$, we have $F_m$ belongs to $C_m ^1(\Nil)$. Moreover
$F_m \in V_m$ if and only if $\overline{F}_m \in V_{-m}$,
where $\overline{F}_{m}$ is the complex conjugate of $F_m$. Therefore Lemma \ref{2*.Tol2} can be applied to each $F_m$ with $m \neq 0$. This part corresponds to the set $\mathcal{A}$. For $m =0$, we have clearly that $e(\xi_1 t)F_0(\Gamma g) \in \mathcal{B}$ for any $\xi_1 \in \mathbb{Z}$. The proof is complete.
\end{proof}

\section{Theorem \ref{1*.MT} for rational $\alpha$}
In this section, we prove Theorem \ref{1*.MT} for rational $\alpha$. In view of Proposition \ref{2*.MP}, we should separately
consider two cases, namely $f\in \mathcal{A}$ and $f \in \mathcal{B}$. The case $f\in \mathcal{B}$ can be reduced
to the case of skew products on $\mathbb{T}^2$ which is already known. The other case $f \in \mathcal{A}$ will be handled by
Fourier analysis and a classical result of Hua.

We begin with skew products on $\mathbb{T}^2$. The following lemma is \cite[Corollary 1.4]{HWY19}.

\begin{LEM}\label{3*.TM-HWY}
Let $\alpha \in \mathbb{R}$ and let $h \colon \mathbb{T} \rightarrow \mathbb{R}$ be $C^\infty$-smooth function.
Define the skew product $T \colon \mathbb{T}^2 \rightarrow \mathbb{T}^2$ by
\begin{equation}\label{3.SP}
T \colon (x,y) \mapsto (x+\alpha, y+h(x)).
\end{equation}
Then the M\"{o}bius Disjointness Conjecture holds for this $(\mathbb{T}^2,T)$.
\end{LEM}

The following is an immediate consequence of Lemma \ref{3*.TM-HWY}.

\begin{COR}\label{3.Cor}
Let $\alpha \in \mathbb{R}$ and let $h_1,h_2 \colon \mathbb{T} \rightarrow \mathbb{R}$ be $C^\infty$-smooth  functions.
Let $T \colon \mathbb{T}^3 \rightarrow \mathbb{T}^3$ be given by
\begin{equation}
T \colon (x,y,z) \mapsto (x+\alpha, y+h_1(x),z+h_2(x)).
\end{equation}
Then the M\"{o}bius Disjointness Conjecture holds for $(\mathbb{T}^3,T)$.
\end{COR}

\begin{proof}
Since $\mathbb{T}^3$ is a compact abelian group, the space of trigonometric polynomials is dense in $C(\mathbb{T}^3)$.
Thus we need only to prove
\begin{equation*}
\lim_{N \rightarrow \infty} \frac{1}{N} \sum_{n=0}^{N-1} \mu(n) f(T^n(x_0,y_0,z_0)) = 0
\end{equation*}
for $(x_0,y_0,z_0) \in \mathbb{T}^3$ and $f(x, y, z)=e(\xi_1 x + \xi_2 y + \xi_3 z )$
where $\xi_1,\xi_2,\xi_3$ are arbitrary integers. For simplicity we write $w_n=f(T^n(x_0,y_0,z_0))$. A direct calculation gives
\begin{equation*}
T^n: (x_0,y_0,z_0) \mapsto \bigg(x_0 + n \alpha, y_0 + \sum_{l=0}^{n-1}h_1(x_0 + \alpha l),
z_0 + \sum_{l=0}^{n-1}h_2(x_0 + \alpha l) \bigg)
\end{equation*}
and hence
\begin{eqnarray}\label{a_n}
w_n=e\bigg(\xi_1 x_0+ \xi_2 y_0 + \xi_3 z_0 + \xi_1 \alpha n + \sum_{l=0}^{n-1}\bigg(\xi_2 h_1(x_0 + \alpha l) + \xi_3 h_2(x_0 + \alpha l)\bigg)\bigg).
\end{eqnarray}

We now construct an analytic skew product $(\mathbb{T}^2,\widetilde{T})$,
in which the sequence $\{w_n\}_{n\geq 1}$ can also be observed. Define $\widetilde{T} : \mathbb{T}^2 \to\mathbb{T}^2$ by
\begin{equation}
\widetilde{T}: (x,y) \mapsto (x+\alpha, y+ \xi_2 h_1(x)+ \xi_3 h_2(x)),
\end{equation}
and let $\widetilde{f}(x,y)=e(\xi_1 x + y) \in C(\mathbb{T}^2)$. Then
\begin{equation*}
\widetilde{f}(\widetilde{T}^n(x_0,0))=
e\bigg(\xi_1 x_0 + \xi_1 \alpha n + \sum_{l=0}^{n-1}\bigg(\xi_2 h_1(x_0 + \alpha l) + \xi_3 h_2(x_0 + \alpha l)\bigg)\bigg).
\end{equation*}
We see that $\widetilde{f}(\widetilde{T}^n(x_0,0))$ differs from $w_n$ by a constant factor only.
Hence the desired result follows from this and Lemma \ref{3*.TM-HWY}.
\end{proof}

\begin{PROP}\label{3*.MP1}
Let $\mathcal{B} \subset C(\TNil)$ be as in Proposition \ref{2*.MP}. Let $T$ be as in Theorem \ref{1*.MT} and let $f \in \mathcal{B}$.
Then, for any $(t_0,\Gamma g_0) \in \TNil$,
\begin{equation*}
\lim_{N \rightarrow \infty} \frac{1}{N}\sum_{n=1}^N \mu(n)f(T^n(t_0,\Gamma g_0))=0.
\end{equation*}
\end{PROP}

\begin{proof}
Let $\widetilde{T} \colon \mathbb{T}^3 \rightarrow \mathbb{T}^3$ be given by
\begin{equation*}
\widetilde{T} \colon (t,x,y) \mapsto (t+\alpha, x+\varphi(t), y+\varphi(t)).
\end{equation*}
Let $\pi$ be the projection of $\TNil$ onto $\mathbb{T}^3$ given by
\begin{equation*}
\pi \colon \lb t, \Gamma \mat{x}{y}{z} \rb \mapsto (t,x,y).
\end{equation*}
Then we have $\pi \circ T = \widetilde{T} \circ \pi$, and hence $(\mathbb{T}^3, \widetilde{T})$ is
a topological factor of $(\TNil, T)$.

Since $f \in \mathcal{B}$, we can write $f(t,\Gamma g)=f_1(t)f_2(\Gamma g)$ for some
$f_1 \in C(\mathbb{T})$ and $f_2 \in H_0 \cap C(\Nil)$. It follows that, for any $z'$,
\begin{equation*}
f_2 \lb \Gamma \mat{x}{y}{z} \rb = f_2 \lb \Gamma \mat{x}{y}{z+z'}\rb.
\end{equation*}
Hence $f_2$ is independent of the $z$-component and induces a well-defined continuous function $\widetilde{f}_2 \in C(\mathbb{T}^2)$ given by
\begin{equation*}
\widetilde{f}_2(x,y)=f_2 \lb \Gamma \mat{x}{y}{z}\rb
\end{equation*}
for any $z \in \mathbb{R}$. Define $\widetilde{f}(t,x,y)\in C(\mathbb{T}^3)$ by
\begin{equation*}
\widetilde{f}(t,x,y)=f_1(t)\widetilde{f}_2(x,y).
\end{equation*}
Then we have $f(t,\Gamma g)=\widetilde{f} \circ \pi(t,\Gamma g)$ for any $(t,\Gamma g) \in \TNil$. Hence
\begin{equation*}
f(T^n(t_0,\Gamma g_0)) = \widetilde{f} \circ \pi \circ T^n(t_0, \Gamma g_0)= \widetilde{f} \circ \widetilde{T}^n \circ \pi (t_0 ,\Gamma g_0)
\end{equation*}
for any $n \geq 1$, and the sequence $\{f(T^n(t_0,\Gamma g_0))\}_{n\geq 1}$ is also
observed in
$(\mathbb{T}^3, \widetilde{T})$. The desired result follows
from Corollary \ref{3.Cor}.
\end{proof}

Now we turn to the case that $f \in \mathcal{A}$. We need the following classical result
of Hua \cite{Hua65}, which is a generalization of Davenport \cite{Dav37}.

\begin{LEM}\label{3.DavLM}
Let $f(x)=\alpha_d x^d + \alpha_{d-1} x^{d-1}+ \ldots +\alpha_1 x + \alpha_0 \in \mathbb{R}[x]$. Let $0 \leq a < q$.
Then, for arbitrary $A>0$,
\begin{equation*}
\sum_{n \leq N \atop {n \equiv a \bmod q }} \mu(n)e(f(n)) \ll_A \frac{N}{\log^A N},
\end{equation*}
where the implied constant may depend on $A,q$ and $d$, but is independent of $\alpha_d, \ldots ,\alpha_0$.
\end{LEM}

\begin{PROP}\label{3*.MP2}
Let $(\TNil,T)$ be as in Theorem \ref{1*.MT} with $\alpha \in \mathbb{Q} \cap [0,1)$.
Let $\mathcal{A}$ be as in Proposition \ref{2*.MP}. Then, for any $(t_0,\Gamma g_0) \in \TNil$,
any $f\in \mathcal{A}$ and any $A>0$,
\begin{equation*}
\sum_{n \leq N} \mu(n)f(T^n(t_0,\Gamma g_0)) \ll_A \frac{N}{\log^A N},
\end{equation*}
where the implied constant depends on $A$ and $\alpha$ only.
\end{PROP}

\begin{proof}
For simplicity, we only consider a typical $f \in \mathcal{A}$ defined by
\begin{equation}\label{DEF/fA}
f\lb t, \Gamma \mat{x}{y}{z}\rb = e(t+x+y+z) \sum_{k \in \mathbb{Z}}e^{-\pi(y+k)^2}e(kx).
\end{equation}
A general $f$ can be treated the same way.

To compute $f(T^n(t_0,\Gamma g_0))$ via \eqref{DEF/fA},
we define, for $t \in \mathbb{T}$ and $n \geq 1$,
\begin{equation*}
S_1(n;t)=\sum_{l=0}^{n-1}\varphi(\alpha l + t),
\quad
S_2(n;t)=\sum_{l=0}^{n-1}\psi(\alpha l + t),
\quad
S_3(n;t)=\sum_{l=0}^{n-1}\varphi^2(\alpha l + t).
\end{equation*}
Also we set $S_1(0;t)=S_2(0;t)=S_3(0;t)=0$ for simplicity. A straightforward calculation gives that
\begin{equation}\label{Tn:t0}
T^n: (t_0,\Gamma g_0)\mapsto (t_0 + n \alpha, \Gamma g_n),
\end{equation}
where
\begin{eqnarray}\label{DEF/gn}
g_n  :=
g_0 \bmat{S_1(n;t_0)}{S_1(n;t_0)}{\frac{1}{2} (S_1(n;t_0))^2 - \frac{1}{2}S_3(n;t_0) + S_2(n;t_0)}.
\end{eqnarray}
Now write
\begin{eqnarray*}
g_0  =  \mat{x_0}{y_0}{z_0}, \quad
g_n  =  \mat{x_n}{y_n}{z_n}
\end{eqnarray*}
where without loss of generality we may assume that $x_0, y_0, z_0 \in [0,1)$,
so that \eqref{DEF/gn} becomes
\begin{equation*}
\begin{cases}
x_n=x_0+S_1(n; t_0), \\
y_n=y_0+S_1(n; t_0), \\
z_n=z_0+\frac{1}{2}(S_1(n; t_0))^2 - \frac{1}{2}S_3(n; t_0) + S_2(n; t_0) +y_0S_1(n; t_0).
\end{cases}
\end{equation*}
Substituting \eqref{Tn:t0} into \eqref{DEF/fA}, we obtain that
\begin{eqnarray}\label{Fur/fTnt0}
f(T^n(t_0,\Gamma g_0))
&=& f\left(t_0 + n \alpha, \Gamma \mat{x_n}{y_n}{z_n}\right) \nonumber\\
&=& e(t_0 + n \alpha +x_n+ y_n+z_n) \sum_{k \in \mathbb{Z}}e^{-\pi(y_n+k)^2}e(kx_n).
\end{eqnarray}

To analyze \eqref{Fur/fTnt0}, we define $w \colon \mathbb{R} \to \mathbb{R}$ by
\begin{equation*}
w(u)=\sum_{k \in \mathbb{Z}} e^{-\pi(u+y_0)^2}e(k(u+x_0)).
\end{equation*}
Then $w$ is an analytic periodic function with period $1$, and hence can be expanded into a Fourier series of the form
\begin{equation*}
w(u)=\sum_{m \in \mathbb{Z}} \widehat{w}(m)e(mu).
\end{equation*}
Plainly
$$
\sum_{m \in \mathbb{Z}} |\widehat{w}(m)| \ll 1,
$$
and the implied constant is absolute. With this function $w$, we can rewrite
\eqref{Fur/fTnt0} as
\begin{eqnarray*}
f(T^n(t_0,\Gamma g_0))
&=& \rho w(S_1(n;t_0)) \\
&& \times e\bigg((y_0+2)S_1(n;t_0)+ \frac{1}{2}(S_1(n;t_0))^2
- \frac{1}{2}S_3(n;t_0) + S_2(n;t_0) +  \alpha n \bigg)
\end{eqnarray*}
where $\rho :=e(t_0+x_0+y_0+z_0)$. By the Fourier expansion of $w$,
\eqref{Fur/fTnt0} finally takes the form
\begin{eqnarray*}\label{fTnt0Gg0}
&& f(T^n(t_0,\Gamma g_0)) \nonumber\\
&&=\rho \sum_{m \in \mathbb{Z}}\widehat{w}(m)e \bigg((y_0+m+2)S_1(n;t_0)
 + \frac{1}{2}(S_1(n;t_0))^2 - \frac{1}{2}S_3(n;t_0) + S_2(n;t_0) +  \alpha n\bigg).
\end{eqnarray*}

Now recall $\alpha \in \mathbb{Q}\cap [0,1)$ in the present situation,
so that we can write $\alpha = a/q$ with $0 \leq a < q$ and $(a,q)=1$.
Thus for any periodic function $h$ with period $1$, we have $h(l_1 \alpha + t_0) = h(l_2 \alpha + t_0)$
whenever $l_1 \equiv l_2 \bmod q$.
For $0 \leq b < q$ and any periodic function $h$ with period $1$, define
\begin{equation*}
\gamma(h,b)=\sum_{l=0}^{b-1} h(l \alpha + t_0)
\end{equation*}
and set $\gamma(h)=\gamma(h,q)/q$. Therefore, for $n \equiv b \bmod q$,
\begin{equation*}
\begin{cases}
S_1(n;t_0)=(n-b)\gamma(\varphi)+\gamma(\varphi,b), \\
S_2(n;t_0)=(n-b)\gamma(\psi)+\gamma(\psi,b), \\
S_3(n;t_0)=(n-b)\gamma(\varphi^2)+\gamma(\varphi^2,b).
\end{cases}
\end{equation*}
It follows from this and the last expression of $f(T^n(t_0,\Gamma g_0))$ that
\begin{equation}\label{3.Eq}
\sum_{n \leq N}\mu(n)f(T^n(t_0,\Gamma g_0)) = \rho \sum_{m \in \mathbb{Z}} \widehat{w}(m)
\sum_{b=0}^{q-1} \sum_{n \leq N \atop {n \equiv b \bmod q }} \mu(n)e(P(n;b)),
\end{equation}
where $P(n;b)$ is a real-valued polynomial in $n$ of degree $\leq 2$ with coefficients depending on $\alpha$, $b$ and $m$. However, by Lemma \ref{3.DavLM}, we have for arbitrary $A>0$ that
\begin{equation*}
\sum_{n \leq N \atop {n \equiv b \bmod q }} \mu(n)e(P(n;b)) \ll_A \frac{N}{\log^{A} N}
\end{equation*}
where the implied constant depending on $q$ (hence on $\alpha$) and $A$ only. Substituting this back to \eqref{3.Eq},
we obtain the desired estimate.
\end{proof}

\begin{PROP}\label{3*.MT}
Theorem \ref{1*.MT} holds for rational $\alpha$.
\end{PROP}

\begin{proof}
The desired result follows from Propositions \ref{2*.MP}, \ref{3*.MP1} and \ref{3*.MP2}.
\end{proof}

\section{Rational approximations of $\alpha$ and further analysis}
From now on, we assume that $\alpha$ is irrational. In this section, we will decompose
$\varphi(t)$, $\varphi^2(t)$ and $\psi(t)$ into the sum of resonant and non-resonant
parts, and investigate them separately. For simplicity we write $\eta(t) \coloneqq \varphi^2(t)$.

Let
\begin{equation*}
\alpha=[0;a_1,a_2, \ldots , a_k, \ldots ]=\frac{1}{a_1+\frac{1}{a_2+\frac{1}{a_3+\ldots}}}
\end{equation*}
be the continued fraction expansion of $\alpha$. This expansion is infinite since $\alpha$ is irrational.
Let $l_k/q_k=[0;a_1,a_2, \ldots ,a_k]$ be the $k$-th convergent of $\alpha$.
Some well-known properties of $l_k/q_k$ are summarized in the following lemma.

\begin{LEM}\label{4*.CF}
Let $\alpha \in [0,1)$ be an irrational number, and $l_k/q_k$ the $k$-th convergent of $\alpha$.

\text{\rm (i)} We have $l_0=0, l_1=1, $ and $l_{k+2}=a_{k+2} l_{k+1}+ l_{k}$ for all $k \geq 0$.
We also have $q_0=1, q_1=a_1$, and $q_{k+2}=a_{k+2} q_{k+1}+ q_{k}$ for all $k \geq 0$.

\text{\rm (ii)} For any $k \geq 1$,
\begin{equation}\label{4.2/qkalp}
\frac{1}{2q_{k+1}} < \| q_k \alpha\| < \frac{1}{q_{k+1}}.
\end{equation}

\text{\rm (iii)} If $|\alpha - l/q | < 1/(2q^2)$ for some integer $l$ and some nonzero integer $q$,
then $l/q=l_k/q_k$ for some $k \geq 1$.
\end{LEM}

Let $\mathcal{Q} =\{q_k \, \colon \, k \geq 1\}$. For $B>2$, define
\begin{equation*}
\mathcal{Q}^{\flat} (B) = \{q_k \in \mathcal{Q} \, \colon \, q_{k+1} \leq q_{k} ^{B}\} \cup \{1\}
\end{equation*}
and
\begin{equation*}
\mathcal{Q}^{\sharp} (B) = \{q_k \in \mathcal{Q} \, \colon \, q_{k+1} > q_{k} ^{B} > 1\}.
\end{equation*}
Furthermore, we define
\begin{equation*}
M_1(B)=\bigcup_{q_k \in \mathcal{Q}^{\sharp} (B)}  \{m \in \mathbb{Z}\, \colon \, q_k \leq |m| < q_{k+1}, \ q_{k} \mid m\} \cup \{0\}
\end{equation*}
and define $M_2(B)=\mathbb{Z} \backslash M_1(B)$.
Now expand $\varphi$ into Fourier series
\begin{equation*}
\varphi(t)=\sum_{m \in \mathbb{Z}} \widehat{\varphi}(m)e(mt),
\end{equation*}
and further decompose $\varphi$ as
\begin{eqnarray}\label{4.phi12}
\varphi(t)
&=&\varphi_1(t)+\varphi_2(t) \nonumber\\
&\coloneqq& \sum_{m \in M_1(B)}\widehat{\varphi}(m)e(mt) + \sum_{m \in M_2(B)}\widehat{\varphi}(m)e(mt).
\end{eqnarray}
We call $\varphi_1$ and $\varphi_2$ the resonant part and the non-resonant part of $\varphi$, respectively.
We can do the same decompositions for $\eta$ and $\psi$, getting
\begin{equation}\label{4.etapsi12}
\eta(t)=\eta_1(t)+\eta_2(t), \quad \psi(t)=\psi_1(t)+\psi_2(t).
\end{equation}
Note that the above decompositions depend on the parameter $B$, though we do not make it explicit.

The following lemma is similar to \cite[Lemma 4.1]{LS15} or \cite[Lemma 5.2]{HWY19}.
We still give the proof here for completeness.
\begin{LEM}\label{4*.M2CONV}
Let $B>2$ and let $\{a(m)\}_{m \in \mathbb{Z}}$ be a sequence such that
\begin{equation}\label{4.BDa(m)}
|a(m)| \ll |m|^{-2B}.
\end{equation}
Then the series
\begin{equation*}
\sum_{m \in M_2(B)} \frac{a(m)}{e(m\alpha)-1}
\end{equation*}
is absolutely convergent.
\end{LEM}

\begin{proof}
By the inequality $|e(x)-1| \asymp \|x\|$ as well as the definition of $M_2(B)$,
it suffices to study
\begin{equation*}
S_1= \sum_{q_k \in \mathcal{Q}} \sum_{q_k \leq |m| < q_{k+1} \atop q_k \nmid m} \frac{|a(m)|}{\|m \alpha\|},
\quad
S_2= \sum_{q_k \in \mathcal{Q}^{\flat}(B)} \sum_{q_k \leq |m| < q_{k+1} \atop q_k \mid m} \frac{|a(m)|}{\|m \alpha\|}.
\end{equation*}

We start with $S_1$. Let $q_k \in \mathcal{Q}$ and let $q_k \leq |m| < q_{k+1}$ with $q_k \nmid m$.
We claim that $\|m \alpha\| \geq 1/(2|m|)$.
Assume on the contrary that $\|m \alpha \| <1/(2|m|)$.
Then there exists $s \in \mathbb{Z}$ such that $|m \alpha -s| <1/(2|m|)$.
Therefore, we have
$$
\bigg|\alpha - \frac{s}{m}\bigg| < \frac{1}{2m^2},
$$
and hence $s/m=l_j/q_j$ for some positive subscript $j$.
So $q_j \mid m$ and we can write $m=aq_j$. Since $|m| < q_{k+1}$, we have $j \leq k$.
But $q_k \nmid m$, and therefore $j < k$.
Finally,
\begin{equation*}
\bigg|\alpha - \frac{l_j}{q_j}\bigg|
= \bigg|\alpha -\frac{s}{m}\bigg|
< \frac{1}{2m^2} \quad \Rightarrow \quad |a||q_j \alpha - l_j|
< \frac{1}{2|m|}
\leq \frac{1}{2},
\end{equation*}
and therefore
\begin{equation*}
\|m \alpha \| =|a| \|q_j \alpha \| > \frac{a}{2q_{j+1}} \geq \frac{1}{2q_k} \geq \frac{1}{2|m|}.
\end{equation*}
This contradiction verifies the claim.

Combing the above claim with (\ref{4.BDa(m)}), we have
\begin{equation*}
\sum_{q_k \leq |m| < q_{k+1} \atop q_k \nmid m} \frac{|a(m)|}{\|m \alpha\|}
\ll \sum_{m\geq q_k} m^{-2B+1}
\ll q_k ^{-2B+2},
\end{equation*}
and hence $S_1$ is absolutely convergent provided $B>2$.

Next we estimate $S_2$.
Let $q_k \in \mathcal{Q}^{\flat}(B)$ and $q_{k} \leq |m| < q_{k+1}$ with $q_k \mid m$. Write $m=dq_k$.
Then $1 \leq |d| \leq q_{k+1}/q_k$, and
\begin{equation}\label{Con/qkalp}
|d|\|q_k \alpha \| \leq \frac{|d|}{q_{k+1}} \leq \frac{1}{q_k} \leq \frac{1}{2}.
\end{equation}
So $\|m \alpha \|$ is actually equal to $|d|\|q_k \alpha\|$. This together with \eqref{4.BDa(m)} and
\eqref{4.2/qkalp} gives
\begin{equation*}
\sum_{q_k \leq |m| < q_{k+1} \atop q_k \mid m} \frac{|a(m)|}{\|m \alpha\|}
\ll \sum_{d\geq 1} (dq_k)^{-2B} (dq_{k+1})
\ll q_{k} ^{-B} \sum_{d\geq 1} d^{-2B+1} \ll q_{k}^{-B},
\end{equation*}
where we have applied $q_{k+1} \leq q_{k}^{B}$. Hence $S_2$ is also absolutely convergent.
The proof is complete.
\end{proof}

Since $\varphi$ is assumed to be $C^{\infty}$-smooth, we have $\widehat{\varphi}(m) \ll |m|^{-2B}$
for any $B>0$. Therefore, by Lemma \ref{4*.M2CONV}, the function $g_{\varphi}$ defined by
\begin{equation*}
g_{\varphi}(t) =  \sum_{m \in M_2(B)} \widehat{\varphi}(m)\frac{e(mt)}{e(m\alpha)-1}
\end{equation*}
is a continuous periodic function with period $1$. Furthermore, we have
\begin{eqnarray}\label{4.gphi}
g_{\varphi}(t+\alpha)-g_{\varphi}(t)
&=&\sum_{m \in M_2(B)} \widehat{\varphi}(m)\frac{e(m(t+\alpha)-mt)}{e(m\alpha)-1} \nonumber\\
&=&\sum_{m \in M_2(B)} \widehat{\varphi} (m)e(mt)
=\varphi_2(t).
\end{eqnarray}
Similarly, there exist continuous periodic functions $g_{\eta}$ and $g_{\psi}$ such that
\begin{equation}\label{4.getapsi}
\eta_2(t)=g_{\eta}(t+\alpha)-g_{\eta}(t),
\quad
\psi_2(t)=g_{\psi}(t+\alpha)-g_{\psi}(t).
\end{equation}

\medskip

Next we investigate the resonant part. For $n \in \mathbb{N}$ and $t \in \mathbb{T}$, define
\begin{equation}\label{4.PHIEPSIdef}
\Phi_n(t)=\sum_{l=0}^{n-1}\varphi_1(l \alpha  + t), \quad
H_n(t)=\sum_{l=0}^{n-1}\eta_1(l \alpha  + t), \quad
\Psi_n(t)=\sum_{l=0}^{n-1}\psi_1(l \alpha + t).
\end{equation}
The following result is essentially \cite[Lemma 4.1]{Wan17}.

\begin{LEM}\label{4*.LEMBDPHI}
Let $B>2$. Then there exists a positive constant $C_1=C_1(B)$ depending on $B$ only, such that
the three inequalities
\begin{equation*}
\begin{cases}
|\Phi_{q_k} (t)-q_k \widehat{\varphi}(0)| \leq C_1 q_{k}^{- B+1},  \\
|H_{q_k} (t)-q_k \widehat{\eta}(0)| \leq C_1 q_{k}^{- B+1}, \\
|\Psi_{q_k} (t)-q_k \widehat{\psi}(0)| \leq C_1 q_{k}^{- B+1}
\end{cases}
\end{equation*}
hold simultaneously for all $t \in \mathbb{T}$ and all $q_k \in \mathcal{Q}^{\sharp}(B)$.
\end{LEM}

\begin{proof}
We only prove the first inequality; proof of the other two is similar.
Fix $q_k \in \mathcal{Q}^{\sharp} (B)$ and $t \in \mathbb{T}$.
We have
\begin{equation*}
\Phi_{q_{k}}(t)
=\sum_{l=0}^{q_k -1}\sum_{m \in M_1(B)} \widehat{\varphi}(m)e(m(l \alpha + t))
=\sum_{m \in M_1(B)}\widehat{\varphi}(m)e(mt)\frac{e(mq_k \alpha)-1}{e(m \alpha) - 1}
\end{equation*}
by interchanging summations. Since $|e(x)-1| \asymp \|x\|$, we need to estimate
\begin{equation}
\sum_{q_j \in \mathcal{Q}^{\sharp}(B)} \sum_{q_j \leq |m| < q_{j+1} \atop q_j \mid m} |\widehat{\varphi}(m)|
 \frac{\| mq_k \alpha\|}{\| m \alpha \|}
\leq
\sum_{q_j \in \mathcal{Q}^{\sharp}(B)} \sum_{d=1}^{{q_{j+1}/q_j}}
|\widehat{\varphi}(dq_j)|  \frac{\| dq_jq_k \alpha\|}{\|  d q_j\alpha \|}.
\end{equation}
Recall that $\varphi$ is $C^{\infty}$-smooth, and therefore $\widehat{\varphi}(m) \ll |m|^{-D}$
for any $D>0$. The value of $D$ will be specified later in terms of $B$. We consider two
cases separately according as $j<k$ or not.

First assume that $j<k$. Similarly to \eqref{Con/qkalp} we can prove $d \|q_j \alpha\| < 1/2$
and hence $\| d q_j\alpha \| = d\|q_j \alpha\|$.
Plainly $\| aq_jq_k \alpha \| \leq dq_j\|q_k \alpha\|$, and hence
\begin{eqnarray*}
|\widehat{\varphi}(dq_j)|  \frac{\| dq_jq_k \alpha\|}{\|  d q_j\alpha \|}
&\ll&
|dq_j|^{-D}  \frac{ dq_j\|q_k \alpha\|}{  d \|q_j\alpha \|}
\ll
(dq_j)^{-D} \frac{q_jq_{j+1}}{q_{k+1}} \\
&\ll&
\frac{q_k}{q_{k+1}} (dq_j)^{-D+1}.
\end{eqnarray*}
It follows that
\begin{eqnarray*}
\sum_{q_j \in \mathcal{Q}^{\sharp}(B) \atop j < k}\sum_{d=1}^{q_{j+1}/q_j}|\widehat{\varphi}(dq_j)|  \frac{\| dq_jq_k \alpha\|}{\|  d q_j\alpha \|}
\ll
\frac{q_k}{q_{k+1}} \sum_{l \geq 1} l^{-D+1} \ll q_k ^{-B+1}
\end{eqnarray*}
provided that $D>3$.

Now we suppose $j \geq k$. We have
\begin{eqnarray*}
|\widehat{\varphi}(d q_j)|  \frac{\| d q_jq_k \alpha\|}{\|  d q_j\alpha \|}
\ll (d q_j)^{-D}q_k
\ll (d q_j)^{-D+1},
\end{eqnarray*}
and hence
\begin{equation*}
\sum_{q_j \in \mathcal{Q}^{\sharp}(B)}\sum_{d=1}^{q_{j+1}/q_j} |\widehat{\varphi}(dq_j)|
\frac{\| aq_jq_k \alpha\|}{\|  a q_j\alpha \|}
\ll
\sum_{l\geq q_k} l^{-D+1}
\ll q_k^{-D+2}
\end{equation*}
which is $\ll q_k^{-B+1}$ provided that $D>B+1$. The first inequality of the lemma now follows on taking
$D=B+3$.
\end{proof}

\section{Measure complexity}
To prove Theorem \ref{1*.MT} for irrational $\alpha$, we will use the concept of
measure complexity introduced in \cite{HWY19}. In this section, we will collect some concepts and facts from \cite{HWY19} without proof.

Let $(X,T)$ be a flow, and $M(X,T)$ the set of all $T$-invariant Borel probability
measures on $X$. A metric  $d$ on $X$ is called compatible if the topology induced by $d$
is the same as the given topology on $X$. For a compatible metric $d$ and an $n \in \mathbb{N}$, define
\begin{equation}\label{Def/bdnx}
\overline{d}_n (x,y) = \frac{1}{n} \sum_{j=0}^{n-1} d(T^j x, T^j y)
\end{equation}
for $x,y \in X$. Then for $\varepsilon >0$ let
$$
B_{\overline{d}_n}(x,\varepsilon) = \{y \in X \, \colon \, \overline{d}_n(x,y) < \varepsilon\},
$$
with which we can further define, for $\rho \in M(X,T)$,
\begin{eqnarray*}
&& s_n(X,T,d,\rho,\varepsilon) \\
&&= \min \bigg\{m \in \mathbb{N}: \exists x_1, \ldots, x_m \in X
\ \text{such that} \ \rho\bigg(\bigcup _{j=1}^m B_{\overline{d}_n}(x_j,\varepsilon)\bigg) > 1 -\varepsilon \bigg\}.
\end{eqnarray*}
Let $(X,d,T,\rho)$ be as above, and let $\{u(n)\}_{n\geq 1}$ be an increasing sequence satisfying $1\leq u(n)\to \infty$ as
$n\to\infty$.
We say that the measure complexity of $(X,d,T,\rho)$ is weaker than $u(n)$ if
\begin{equation*}
\liminf_{n \to \infty} \frac{s_n(X,T,d,\rho,\varepsilon)}{u(n)}=0
\end{equation*}
for any $\varepsilon >0$. In view of Lemma \ref{5*.2} below,
this property is independent of the choice of compatible metrics. Hence we can say instead that
the measure complexity of $(X,T,\rho)$ is weaker than $u(n)$. We say the
measure complexity of $(X,T,\rho)$ is \emph{sub-polynomial} if the measure
complexity of $(X,T,\rho)$ is weaker than $n^{\tau}$ for any $\tau >0$. We are going to need
the following result, which is the main theorem of \cite{HWY19}.

\begin{LEM}\label{5*.HWYTM}
If the measure complexity of $(X,T,\rho)$ is sub-polynomial for any $\rho \in M(X,T)$,
then the M\"{o}bius Disjointness Conjecture holds for $(X,T)$.
\end{LEM}

We explain the number theory behind Lemma \ref{5*.HWYTM}. The measure complexity defined above
can be viewed as an averaged form of entropy, and it is well-known that Chowla's conjecture implies
the M\"obius Disjointness Conjecture. In \cite{MRT15}, Matom\"{a}ki, Radziwill and Tao
established an averaged form of Chowla's conjecture. This allows to use the measure
complexity defined here, rather than the original topological entropy, to investigate the M\"{o}bius disjointness.

\medskip

Let $(X,T)$ and $(Y,S)$ be two flows, and $d$ and $d'$ the metrics on $X$ and $Y$
respectively. Let $\rho \in M(X,T)$ and $\nu \in M(Y,S)$. Let $\mathcal{B}_X$ and $\mathcal{B}_Y$ be the
Borel $\sigma$-algebras of $X$ and $Y$ respectively. We say $(X,\mathcal{B}_X,T,\rho)$ is measurably
isomorphic to $(Y,\mathcal{B}_Y,S,\nu)$, if there exist $X' \subset X$, $Y' \subset Y$
with $\rho(X')=\rho(Y')=1$ and $TX' \subset X'$, $SY' \subset Y'$, and an invertible
measure-preserving map $\phi: X' \rightarrow Y'$ such that $\phi \circ T (x) = S \circ \phi(x)$
for any $x \in X'$. The following proposition is \cite[Proposition 2.2]{HWY19}, which is
important when calculating the measure complexity.

\begin{LEM}\label{5*.2}
Let $\{u(n)\}_{n\geq 1}$ be an increasing sequence satisfying $1\leq u(n)\to \infty$ as
$n\to\infty$. Assume that $(X,\mathcal{B}_X,T,\rho)$ is measurably isomorphic to $(Y, \mathcal{B}_Y, S, \nu)$.
Then the measure complexity of $(X,d,T,\rho)$ is weaker than $u(n)$ if and only if the measure complexity of $(Y, d', S, \nu)$ is weaker than $u(n)$.
\end{LEM}

\section{Theorem \ref{1*.MT} for irrational $\alpha$}

The purpose of this section is to prove the next result.

\begin{PROP}\label{6*.MT}
Let $(\TNil,T)$ be as in Theorem \ref{1*.MT} with $\alpha$ irrational. Then the measure complexity of $(\TNil, T, \rho)$ is sub-polynomial for any $\rho \in M(\TNil,T)$.
\end{PROP}
Before proving Proposition \ref{6*.MT}, we need to choose a proper metric on $\TNil$.
The following facts can be found in Sections 2 and 5 in Green-Tao \cite{GT12a}, which we will directly state
without proof. A more detailed version is given in Appendix I.
The lower central series filtration $G_{\bullet}$ on $G$ is the sequence of closed connected subgroups
\begin{equation*}
G= G_1 \supseteq G_2 \supseteq G_3= \{\mathrm{id}_G\}
\end{equation*}
where
\begin{equation*}
G_2=[G_1,G]=\mat{0}{0}{\mathbb{R}}
\end{equation*}
and $\mathrm{id}_G$ is the identity element of $G$.
Let $\mathfrak{g}$ be the Lie algebra of $G$. Let
\begin{equation*}
X_1=\liemat{1}{0}{0}, \quad X_2= \liemat{0}{1}{0}, \quad X_3 = \liemat{0}{0}{1}.
\end{equation*}
Then $\mathcal{X}=\{X_1,X_2,X_3\}$ is a Mal'cev basis adapted to $G_{\bullet}$.
The corresponding Mal'cev coordinate map $\kappa: G \rightarrow \mathbb{R}^3$ is given by
\begin{equation}\label{5.MC}
\kappa \mat{x}{y}{z}= (x,y,z-xy).
\end{equation}
The metric $d_G$ on $G$ is defined to be the largest metric such that $d_G(g_1,g_2) \leq |\kappa(g_1 ^{-1} g_2)|$, where $|\cdot|$ is the $l^{\infty}$-norm on $\mathbb{R}^3$. This metric can be more explicitly expressed as
\begin{equation*}
d_G(g_1,g_2)=\inf\bigg\{\sum_{i=0}^{n-1}\min(|\kappa(h_{i-1}^{-1}h_i)|,
|\kappa(h_i ^{-1}h_{i-1})|) \colon h_0,\ldots,h_n \in G ; h_0=g_1, h_n = g_2 \bigg\},
\end{equation*}
from which we can see that $d_G$ is left-invariant.
By (\ref{5.MC}), we have
\begin{equation}\label{5.MCUB}
\left|\kappa\mat{x}{y}{z} \right| \leq |x|+|y|+|z|
\end{equation}
provided that $x,y \in [0,1)$. The above metric on $G$ descends to a metric on $\Nil$ given by
\begin{equation*}
d_{\Nil}(\Gamma g_1, \Gamma g_2) \coloneqq \inf\{d_{G}(g_1', g_2') \, \colon \, g_1', g_2' \in G, \Gamma g_1=\Gamma g_1', \Gamma g_2 = \Gamma g_2 '\}.
\end{equation*}
It can be proved that $d_{\Nil}$ is indeed a metric on $\Nil$. Since $d_{G}$ is left-invariant, we also have
\begin{equation}\label{6.dNil<dG}
d_{\Nil}(\Gamma g_1, \Gamma g_2)= \inf_{\gamma \in \Gamma} d_{G}(g_1,\gamma g_2).
\end{equation}
Finally, we take $d_{\mathbb{T}}$ to be the canonical Euclidean metric on $\mathbb{T}$, and $d=d_{\TNil}$ the $l^{\infty}$-product metric of $d_{\mathbb{T}}$ and $d_{\Nil}$ given by
\begin{equation}\label{6.DefMetric}
d((t_1,\Gamma g_1),(t_2,\Gamma g_2))= \max ( d_{\mathbb{T}}(t_1,t_2),  d_{\Nil}(\Gamma g_1,\Gamma g_2) ).
\end{equation}
In view of Lemma \ref{5*.2}, the choice of compatible metrics does not
affect the measure complexity. Thus the above choice of $d$ is admissible.

\begin{proof}[Proof of Proposition \ref{6*.MT}]
Fix $\tau> 0$. We want to show that, for any $\varepsilon > 0$,
\begin{equation*}
\liminf_{n \rightarrow \infty} \frac{s_n(\TNil,T,d,\rho,\varepsilon)}{n^{\tau}}=0.
\end{equation*}
Without loss of generality, we assume that both $\tau$ and $\varepsilon$ are less than $10^{-2}$,
and also both $\varepsilon^{-1}$ and $\tau^{-1}$ are integers.
Set $B=8\tau^{-1}+1$. Let $\mathcal{Q}^{\flat}(B), \mathcal{Q}^{\sharp}(B), M_1(B),
M_2(B), g_{\varphi}, g_{\eta}$ and $g_{\psi}$ be as in Section 4.

We first assume that $\mathcal{Q}^{\flat} (B)$ is infinite. Construct a transformation $S \colon \TNil \rightarrow \TNil$ as
\begin{equation}\label{6.DEFS}
S:
(t,\Gamma g)
\mapsto
\left (t, \Gamma g \bmat{g_{\varphi}(t)}{g_{\varphi}(t)}{\frac{1}{2}g^2_{\varphi}(t) - \frac{1}{2}g_{\eta}(t)+g_{\psi}(t)} \right).
\end{equation}
Write $T_1 = S^{-1} \circ T \circ S$. Then a straightforward calculation gives
\begin{eqnarray*}
&& T_1(t,\Gamma g)= S^{-1} \circ T \circ S (t, \Gamma g) \\
&& = S^{-1} \circ T \left( t,\Gamma g
    \bmat{g_{g_{\varphi}(t)}}{g_{\varphi}(t)}{\frac{1}{2}g^2_{\varphi}(t) - \frac{1}{2}g_{\eta}(t)+g_{\psi}(t)} \right) \\
&& = S^{-1}\left( t+\alpha, \Gamma g
    \bmat{\varphi(t)+g_{\varphi(t)}}{\varphi(t)+g_{\varphi(t)}}
    {\frac{1}{2}g^2_{\varphi}(t) - \frac{1}{2}g_{\eta}(t)+g_{\psi}(t) + \psi(t) + g_{\varphi}(t)\varphi(t)}\right) \\
&& =\left(t+\alpha, \Gamma g
\bmat{\varphi(t)+g_{\varphi}(t)-g_{\varphi}(t+\alpha)} {\varphi(t)+g_{\varphi}(t)-g_{\varphi}(t+\alpha)}
{\varpi}\right),
\end{eqnarray*}
where we have written temporarily
\begin{eqnarray*}
\varpi:=
\frac{1}{2}g^2_{\varphi}(t) - \frac{1}{2}g_{\eta}(t)+g_{\psi}(t) + \psi(t) + g_{\varphi}(t)\varphi(t)+\frac{1}{2}g^2_{\varphi}(t+\alpha) \\
 + \frac{1}{2}g_{\eta} (t+\alpha) -g_{\psi}(t+\alpha) - g_{\varphi}(t)g_{\varphi}(t+\alpha) - \varphi(t)g_{\varphi}(t+\alpha).
\end{eqnarray*}
Let $\varphi_1,\varphi_2,\eta_1,\eta_2,\psi_1,\psi_2$ be as in (\ref{4.phi12}) and (\ref{4.etapsi12}).
By (\ref{4.gphi}), (\ref{4.getapsi}), as well as $\varphi^2(t)=\eta(t)$, the above $\varpi$ can be simplified
\footnote{A careful reader will observe that the simplification of $\varpi$ in \eqref{6/SimPi} works
for $T$ satisfying \eqref{1.TForm}, but not
for general $S$ of the form \eqref{1/Gen/S}. This is the point where the exact form of \eqref{1.TForm} is indeed
needed.
}
as
\begin{eqnarray}\label{6/SimPi}
\varpi
&=& \frac{1}{2}(g_{\varphi}(t+\alpha)-g_{\varphi}(t))^2 - \varphi(t)(g_{\varphi}(t+\alpha)-g_{\varphi}(t))
           + \frac{1}{2}(g_{\eta}(t+\alpha)-g_{\eta}(t)) + \psi_1(t) \nonumber\\
&=& \frac{1}{2}\varphi^2 _2(t)- \varphi(t)\varphi_2(t)+\frac{1}{2}\eta_2(t)+\psi_1(t) \nonumber\\
&=& \frac{1}{2}(\varphi(t)-\varphi_2(t))^2 - \frac{1}{2}\varphi^2(t)+\frac{1}{2}\eta(t)-\frac{1}{2}\eta_1(t)+\psi_1(t) \nonumber\\
&=& \frac{1}{2}\varphi_1 ^2 (t) -\frac{1}{2}\eta_1(t) + \psi_1(t).
\end{eqnarray}
It follows that
\begin{equation}\label{6.T1}
T_1: (t,\Gamma g)
\mapsto
\left(t+\alpha, \Gamma g \bmat{\varphi_1(t)}{\varphi_1(t)}{\frac{1}{2}\varphi_1 ^2 (t) -\frac{1}{2}\eta_1(t) + \psi_1(t)}\right),
\end{equation}
and by induction on $n$,
\begin{equation*}
T^n_1: (t,\Gamma g)
\mapsto
\left(t+n\alpha , \Gamma g
\bmat{\Phi_n(t)}{\Phi_n(t)}{\frac{1}{2}\Phi_n ^2 (t) - \frac{1}{2}H_n(t) + \Psi_n(t)}\right),
\end{equation*}
where $\Phi_n(t)$, $H_n(t)$ and $\Psi_n(t)$ are as in (\ref{4.PHIEPSIdef}). Clearly, $S$ is a homeomorphism on $\TNil$.
Hence by Lemma \ref{5*.2} we need only to show that the measure complexity of $(\TNil,T_1,\nu)$ is
weaker than $n^{\tau}$, where $\nu = \rho \circ S$.

Let $C_1=C_1(B)>0$ be the constant in Lemma \ref{4*.LEMBDPHI}. The functions $\varphi_1(t)$, $\eta_1(t)$
and $\psi_1(t)$ are Lipschitz continuous, and therefore there exists $L>0$ such that
\begin{equation}\label{6.Lip}
\begin{cases}
|\varphi_1(t_1)-\varphi_1(t_2)|\leq L \|t_1-t_2\|,  \\
|\eta_1(t_1)-\eta_1(t_2)|\leq L \|t_1-t_2\|,  \\
|\psi_1(t_1)-\psi_1(t_2)| \leq L \|t_1-t_2\|
\end{cases}
\end{equation}
for any $t_1$, $t_2 \in \mathbb{T}$. We also assume that $L$ is large enough such that $L>\varepsilon^{-1}$.
Moreover, since $\varphi_1(t)$, $\eta_1(t)$ and $\psi_1(t)$ are continuous, there exists a constant $C_2>0$ such that
\begin{equation*}
|\varphi_1(t)|\leq C_2, \ |\eta_1(t)|\leq C_2, \ |\psi_1(t)|\leq C_2
\end{equation*}
for all $t \in \mathbb{T}$. Since $q_k \rightarrow \infty$ as $k \rightarrow \infty$, there exists $K_0>0$ such that
$(C_1+C_2)/q_k < \varepsilon$ for all $k \geq K_0$. For $k \geq K_0$, define
\begin{equation*}
F_1(k)=\left\{t=\frac{j\varepsilon}{Lq_k } \in \mathbb{T} \, \colon \, j=0,1, \ldots, \frac{Lq_k}{\varepsilon}-1 \right\}
\end{equation*}
and
\begin{equation*}
F_2(k)=\left\{ \Gamma g = \mat{j_1(q_k ^2 L)^{-1}}{j_2(q_k ^2 L)^{-1}}{j_3(q_k ^2 L)^{-1}}\in \Nil \, \colon \, j_1, \, j_2,\, j_3 = 0,1, \ldots , q_k ^2 L-1\right\}.
\end{equation*}
Let
\begin{equation*}
F(k)=\{(t,\Gamma g) \in \TNil \, \colon \, t \in F_1(k), \ \Gamma g \in F_2(k)\}.
\end{equation*}
Then $\# F(k) = \varepsilon^{-1}L^4q_k^7$.

Now assume that $q_k \in \mathcal{Q}^{\sharp}(B)$ with $k \geq K_0$, and set
\begin{equation}\label{Def/nk}
n_k=q_k ^{B-1}.
\end{equation}
Then any positive integer $m\leq n_k$ can be uniquely written as
\begin{equation}\label{6.i}
m=a_mq_k+b_m
\end{equation}
with $0 \leq b_m < q_k$ and $a_m \leq q_k ^{B-2}$. By the definition of $F(k)$, for any
$$
(t,\Gamma g)=\left(t,\Gamma \mat{x}{y}{z}\right) \in \TNil
$$
with $x$, $y$, $z \in [0,1)$, there exists
$$
(t^*,\Gamma g^*)=\left(t^*, \Gamma \mat{x^*}{y^*}{z^*} \right) \in F(k)
$$
such that $\|t-t^*\| \leq \varepsilon/(Lq_k)$
and
\begin{equation}\label{6.maxXYZ}
\max \{ \, |x-x^*|, \ |y-y^*|, \ |z-z^*| \, \} \leq \frac{1}{q_k ^2L}.
\end{equation}
We want to show that $d(T_1^m(t,\Gamma g), T_1 ^m(t^*,\Gamma g^*))$ is small for any $m \leq n_k$
where $n_k$ is as in \eqref{Def/nk}.

Let
\begin{equation*}
Y(m)=\bmat{\Phi_m(t)}{\Phi_m(t)}{\frac{1}{2}\Phi_m^2(t) - \frac{1}{2}H_m (t) + \Psi_m(t)}
\end{equation*}
and
\begin{equation*}
Y^*(m)=\bmat{\Phi_m(t^*)}{\Phi_m(t^*)}{\frac{1}{2}\Phi_m^2(t^*) - \frac{1}{2}H_m (t^*) + \Psi_m(t^*)}.
\end{equation*}
Then we have
\begin{equation*}
T^m(t,\Gamma g)=(t+\alpha m, \Gamma gY(m)), \quad  T^m(t^*,\Gamma g^*)=(t^*+\alpha m, \Gamma g^*Y^*(m)).
\end{equation*}
Therefore, by our choice of the metric on $\TNil$, we have
\begin{equation}\label{6.d<dT+dNil}
d(T^m(t,\Gamma g),T^m(t^*,\Gamma g^*)) \leq \max (\|t-t^*\|, d_{\Nil}(\Gamma gY(m),\Gamma g^*Y^*(m)) ).
\end{equation}

The term $\|t-t^*\|$ can be arbitrarily small as $q_k\to\infty$. So it remains to bound the last term in \eqref{6.d<dT+dNil}.
By the triangle inequality and (\ref{6.dNil<dG}),
\begin{eqnarray}\label{6.dNil<dG+}
&& d_{\Nil}(\Gamma gY(m),\Gamma g^*Y^*(m)) \nonumber\\
&& \leq
d_{\Nil}(\Gamma g^*Y(m),\Gamma gY(m))+ d_{\Nil}(\Gamma g^* Y^*(m),\Gamma g^* Y(m)) \nonumber\\
&& \leq
d_G(g^*Y(m),gY(m))+d_G(g^* Y^*(m),g^* Y(m)) \nonumber\\
&& =
d_G(g^*Y(m),gY(m))+d_G(Y^*(m),Y(m)),
\end{eqnarray}
where the last equality follows from the left invariance of $d_G$. Furthermore, by the definition of $d_G$, we have
\begin{eqnarray}\label{6.dG<kappa}
\begin{cases}
d_G(g^*Y(m),gY(m)) \leq |\kappa(Y(m)^{-1}g^{-1}g^*Y(m))|, \\
d_G(Y^*(m),Y(m)) \leq |\kappa(Y(m)^{-1}Y^*(m))|,
\end{cases}
\end{eqnarray}
where $\kappa$ is the Mal'cev coordinate map defined by (\ref{5.MC}) and $|\cdot|$ is the $l^{\infty}$-norm on $\mathbb{R}^3$.

A straightforward calculation gives
\begin{eqnarray*}
&& Y(m)^{-1}g^{-1}g^*Y(m) \\
&& = \bmat{-\Phi_m(t)}{-\Phi_m(t)}{\frac{1}{2}\Phi^2_m(t) + \frac{1}{2}H_m(t)-\Psi_m(t)}
  \bmat{x^*-x}{y^*-y}{z^* - z +xy - xy^*} \\
&& \quad \times
\bmat{\Phi_m(t)}{\Phi_m(t)}{\frac{1}{2}\Phi^2_m(t) - \frac{1}{2}H_m(t)+\Psi_m(t)} \\
&& =\bmat{x^*-x}{y^* - y}{(z^*-z)+x(y-y^*)-\Phi_m(t)(x-x^*-y+y^*)}.
\end{eqnarray*}
Since $x,y\in [0,1)$, by (\ref{5.MCUB}), we have
\begin{equation*}
|\kappa(Y(m)^{-1}g^{-1}g^*Y(m))| \leq (\Phi_m(t)+2)(|x-x^*|+|y-y^*|+|z-z^*|).
\end{equation*}
By Lemma \ref{4*.LEMBDPHI}, we have
\begin{equation*}
|\Phi_{q_k}(t)-q_k \widehat{\varphi}(0)| \leq C_1 q_k ^{-B+1}.
\end{equation*}
However, by the assumption (\ref{1.ASSMFORphi}), the Fourier coefficient $\widehat{\varphi}(0)$ is zero, and therefore
$$
|\Phi_{q_k}(t)| \leq C_1 q_k ^{-B+1}.
$$
Hence by the definition of $\Phi_n(t)$ and (\ref{6.i}), we obtain
\begin{eqnarray*}
|\Phi_m(t)|
&\leq& \sum_{r=0}^{a_m-1}|\Phi_{q_k}(t+rq_k\alpha)|+\sum_{l=0}^{b_m}|\varphi_1(t+(rq_k+l)\alpha)| \\
&\leq& \frac{C_1a_m}{q_k ^{B-1}}+C_2 q_k \leq \frac{C_1}{q_k}+C_2q_k.
\end{eqnarray*}
Thus by (\ref{6.maxXYZ}), (\ref{6.dNil<dG+}) and (\ref{6.dG<kappa}), we obtain
\begin{eqnarray}\label{6.d_G1BD}
d_G(g^*Y(m), gY(m))
&\leq& (\Phi_m(t)+2)(|x-x^*|+|y-y^*|+|z-z^*|) \nonumber\\
&\leq& \frac{3(3+C_2 q_k)}{Lq_k ^2} \leq \frac{9}{L}+\frac{3C_2}{q_k} < 12 \varepsilon.
\end{eqnarray}
The treatment of $d_G(Y^*(m),Y(m))$ is similar. We calculate that
\begin{eqnarray*}
Y(m)^{-1}Y^*(m)
&=& \bmat{-\Phi_m(t)}{-\Phi_m(t)}{\frac{1}{2}\Phi^2_m(t)+\frac{1}{2}H_m(t) - \Psi_{m}(t)}\\
&& \times \bmat{\Phi_m(t^*)}{\Phi_m(t^*)}{\frac{1}{2}\Phi^2_m(t^*)-\frac{1}{2}H_m(t^*) + \Psi_{m}(t^*)} \\
&=& \bmat{\Phi_m(t^*)-\Phi_m(t)}{\Phi_m(t^*)-\Phi_m(t)}
       {\varpi},
\end{eqnarray*}
where we have written temporarily
$$
\varpi :=\frac{1}{2}(\Phi_m(t^*)-\Phi_m(t))^2 + \frac{1}{2}(H_m(t)-H_m(t^*))-(\Psi_{m}(t^*)-\Psi_{m}(t)).
$$
By Lemma \ref{4*.LEMBDPHI}, (\ref{6.Lip}) and (\ref{6.i}), we have
\begin{eqnarray*}
|\Phi_m(t^*)-\Phi_m(t)|
&\leq &\sum_{r=0}^{a_m-1}(| \Phi_{q_k}(t^*+rq_k \alpha) - q_k \widehat{\varphi}(0)|+|\Phi_{q_k}(t+rq_k \alpha) - q_k \widehat{\varphi}(0)|)\\
     && +\sum_{l=0}^{b_m}|\varphi_1(t^* + (a_m q_k + l)\alpha) - \varphi_1(t + (a_m q_k + l)\alpha)| \\
&\leq& \frac{C_1}{q_k}+q_kL \|t-t^{*}\| < 2\varepsilon.
\end{eqnarray*}
The same estimate holds for $|H_m(t^*)-H_m(t)| $ and $|\Psi_m(t^*)-\Psi_m(t)|$ as well.
Now since $|\Phi_m(t^*)-\Phi_m(t)|<1$, we can apply (\ref{5.MCUB}) to $\kappa(Y(m)^{-1}Y^*(m))$, getting
\begin{equation}\label{6.d_G2BD}
d_G(Y^*(m),Y(m)) \leq |\kappa(Y(m)^{-1}Y^*(m))| < 8 \varepsilon.
\end{equation}

From (\ref{6.d<dT+dNil}), (\ref{6.dNil<dG+}), (\ref{6.d_G1BD}) and (\ref{6.d_G2BD}), we conclude that
\begin{equation*}
d(T^m(t,\Gamma g),T^m(t^*,\Gamma g^*)) < 20 \varepsilon
\end{equation*}
for all $m \leq n_k$. Here, and in what follows, $n_k$ is as in \eqref{Def/nk}.
Thus, by \eqref{Def/bdnx},
\begin{equation*}
\overline{d}_{n_k} (T^m(t,\Gamma g),T^m(t^*,\Gamma g^*)) = \frac{1}{n_k} \sum_{m=0}^{n_k-1} d(T^m(t,\Gamma g),T^m(t^*,\Gamma g^*)) < 20 \varepsilon.
\end{equation*}
This means that $\TNil$ can be covered by $\#F(k) = \varepsilon^{-1}L^4q_k^7$ balls of radius $20\varepsilon$ under the metric $\overline{d}_{n_k}$ since $(t,\Gamma g)$ can be chosen arbitrarily. It follows that
\begin{equation*}
s_{n_k}(\TNil,T_1,d,\nu,20\varepsilon) \leq \varepsilon^{-1}L^4q_k^7.
\end{equation*}
Since $\mathcal{Q}^{\sharp}(B)$ is infinite, we can let $q_k$ tend to infinity along $\mathcal{Q}^{\sharp}(B)$, getting
\begin{eqnarray*}
&&\liminf_{n \rightarrow \infty} \frac{s_n(\TNil,T_1,d,\nu,20\varepsilon)}{n^{\tau}}
\leq
\liminf_{{k \rightarrow \infty} \atop {q_k \in \mathcal{Q}^{\sharp}(B), k \geq K_0}} \frac{s_{n_k}(\TNil,T_1,d,\nu,20\varepsilon)}{n_k^{\tau}}\\
&&\leq \liminf_{{k \rightarrow \infty} \atop {q_k \in \mathcal{Q}^{\sharp}(B), k \geq K_0}} \frac{\varepsilon^{-1}L^4q_k^7}{q_k ^{8+\tau}}
\leq
\liminf_{{k \rightarrow \infty} \atop {q_k \in \mathcal{Q}^{\sharp}(B), k \geq K_0}} \frac{\varepsilon^{-1}L^4}{q_k} = 0.
\end{eqnarray*}
Since $\varepsilon$ can be arbitrarily small, this means that the measure complexity of $(\TNil, T, \rho)$ is weaker that $n^{\tau}$ when $\mathcal{Q}^{\sharp}(B)$ is infinite.

Finally, we deal with the case that $\mathcal{Q}^{\sharp}(B)$ is finite. Now $M_1(B)$ is also finite. Hence the conclusion of Lemma \ref{4*.M2CONV} still holds if we replace $M_2(B)$ by $\mathbb{Z} \backslash \{0\}$. Hence the functions $\widetilde{g}_{\varphi}(t)$, $\widetilde{g}_{\eta}(t)$ and $\widetilde{g}_{\psi}(t)$ defined by
\begin{equation*}
\widetilde{g}_{\varphi}(t) = \sum_{m \neq 0} \widehat{\varphi}(m)\frac{e(mt)}{e(m\alpha)-1},
\,
\widetilde{g}_{\eta}(t) = \sum_{m \neq 0} \widehat{\eta}(m)\frac{e(mt)}{e(m\alpha)-1},
\,
\widetilde{g}_{\psi}(t) = \sum_{m \neq 0} \widehat{\psi}(m)\frac{e(mt)}{e(m\alpha)-1}
\end{equation*}
are continuous and periodic with period one. Thus we can write
\begin{equation}
\begin{cases}
\varphi(t)=\widetilde{g}_{\varphi}(t+\alpha)-\widetilde{g}_{\varphi}(t), \\
\eta(t)=\widehat{\eta}(0)+\widetilde{g}_{\eta}(t+\alpha)-\widetilde{g}_{\eta}(t), \\
\psi(t)=\widehat{\psi}(0)+\widetilde{g}_{\psi}(t+\alpha)-\widetilde{g}_{\psi}(t).
\end{cases}
\end{equation}
Notice that $\widehat{\varphi}(0)=0$, and so there are no constant terms in the first equation. Similarly to (\ref{6.DEFS}), we define $\widetilde{S} \colon \TNil \rightarrow \TNil$ by
\begin{equation*}
\widetilde{S}: (t,\Gamma g) \mapsto
\left (t, \Gamma g
    \bmat{\widetilde{g}_{\varphi}(t)}{\widetilde{g}_{\varphi}(t)}{\frac{1}{2}\widetilde{g}^2_{\varphi}(t) - \frac{1}{2}\widetilde{g}_{\eta}(t)+\widetilde{g}_{\psi}(t)} \right).
\end{equation*}
Then $\widetilde{T}_1 \coloneqq \widetilde{S}^{-1} \circ T \circ \widetilde{S}$ is given by
\begin{equation}
\widetilde{T}_1: (t,\Gamma g)
\mapsto
\left(t+\alpha, \Gamma g \bmat{0}{0}{-\frac{1}{2}\widehat{\eta}(0) + \widehat{\psi}(0)}\right)
\end{equation}
as in (\ref{6.T1}). Again by Lemma \ref{5*.2}, the measure complexity of $(\TNil, T, \rho)$ is weaker that $n^{\tau}$ if and only if the measure complexity of $(\TNil, \widetilde{T}_1, \nu)$ is weaker that $n^{\tau}$, where $\nu=\rho \circ S$. However, $d$ is invariant under $\widetilde{T}_1$. So we have for any $n \geq 1$ and any $\varepsilon>0$ that
\begin{equation*}
s_n(\TNil, \widetilde{T}_1, d, \nu, \varepsilon)= s_1(\TNil, \widetilde{T}_1, d, \nu, \varepsilon).
\end{equation*}
Since $\TNil$ is compact, we have $s_1(\TNil, \widetilde{T}_1, d, \nu, \varepsilon) < \infty$
and consequently
\begin{equation*}
\lim_{n \rightarrow \infty} \frac{s_n(\TNil, \widetilde{T}_1, d, \nu, \varepsilon)}{n^{\tau}}
=
\lim_{n \rightarrow \infty} \frac{s_1(\TNil, \widetilde{T}_1, d, \nu, \varepsilon)}{n^{\tau}}
=0.
\end{equation*}
Hence the measure complexity of $(\TNil, T, \rho)$ is also weaker than $n^{\tau}$ if $\mathcal{Q}^{\sharp}(B)$ is finite. The proof is complete.
\end{proof}
\begin{proof}[Proof of Theorem \ref{1*.MT}]
Theorem \ref{1*.MT} follows from Proposition \ref{3*.MT}, Lemma \ref{5*.HWYTM} and Proposition \ref{6*.MT}.
\end{proof}

\section{Appendix I: preliminaries on nilmanifolds and the Mal'cev basis}

\Def{Nilmanifold}{
Let $G$ be a connected, simply connected Lie group. The identity element of $G$ is denoted by $\mathrm{id}_G$.  A \emph{filtration} $G_{\bullet}$ on $G$ is a sequence of closed connected subgroups
\begin{equation*}
G=G_0=G_1 \supseteq G_2 \supseteq \cdots \supseteq G_d \supseteq G_{d+1}=\{\mathrm{id}_G\}
\end{equation*}
satisfying $[G_i,G_j] \subset G_{i+j}$ for all integers $i,j \geq 0$. The \emph{degree} of $G_{\bullet}$ is the least integer $d$ for which $G_{d+1}=\{\mathrm{id}_G\}$ where $[H,K]$ is the commutator group of $H$ and $K$. If $G$ possesses a filtration, we say that $G$ is \emph{nilpotent}. Let $\Gamma$ be a discrete cocompact subgroup of $G$. Then the quotient space $\Nil$
 is called a \emph{nilmanifold}. The \emph{dimension} of $\Nil$ is defined to be the dimension of $G$. }

We remark that, in the literature, left coset form of the nilmanifold $G / \Gamma$ is more commonly used; see for example
 \cite{GT12a}. We use the right coset form here in order to directly apply the results in \cite{Tol77}.
 Certainly, there is no essential difference between these two forms. But one should carefully
 modify the definition of the metric on $\Nil$ defined below.

\Def{Lower central series filtration}{
Let $G$ be a nilpotent Lie group possessing a filtration of degree $d$. Then the sequence $\{G_i\}$ defined by $G=G_0=G_1$ and $G_{i+1}=[G_i,G]$ terminates with $G_{s+1}=\{\mathrm{id}_G\}$ for some integer $s \leq d$. This sequence is called the \emph{lower central series filtration} of $G$ and the least integer $s$ is called the \emph{step} of $G$.
}

For a Lie group $G$ with Lie algebra $\mathfrak{g}$, one can define the exponential map
$\exp \colon \mathfrak{g} \rightarrow G$ and the logarithm map $\log \colon G \rightarrow \mathfrak{g}$.
When $G$ is a connected, simply connected nilpotent group, both these two maps are diffeomorphisms.

\Def{Mal'cev basis}{
Let $G$ be a $m$-dimensional $s$-step nilpotent Lie group with the lower central series filtration $G_{\bullet}$. Let $\Gamma$ be a discrete cocompact subgroup. A basis $\mathcal{X}=\{X_1,\ldots,X_m\} \subset \mathfrak{g}$ is called a \emph{Mal'cev basis} for $\Nil$ adapted to $G_{\bullet}$ if it satisfies the following conditions: \\
\indent \text{\rm (i)} For each $j=0, \ldots , m-1$, the subspace $\mathfrak{h}_j \coloneqq \Span(X_{j+1}, \ldots , X_m)$ is a Lie algebra ideal of $\mathfrak{g}$. Therefore, the group $H_j \coloneqq \exp \mathfrak{h}_j$ is a normal subgroup of $G$; \\
\indent \text{\rm (ii)}  For every $0 \leq i \leq s$, $G_i = H_{m-m_i}$ where $m_i$ is the dimension of $G_i$; \\
\indent \text{\rm (iii)}  Each $g \in G$ can be uniquely written as $\exp(t_1X_1)\exp(t_2X_2)\ldots \exp(t_mX_m)$ for some $t_1, \ldots , t_m \in \mathbb{R}$; \\
\indent \text{\rm (iv)} The discrete cocomapct subgroup $\Gamma$ is given by
\begin{equation*}
\Gamma = \{g=\exp(t_1X_1)\exp(t_2X_2)\ldots \exp(t_mX_m)\in G \, \colon \, t_1, \ldots , t_m \in \mathbb{Z}\}.
\end{equation*}
}

By the result of Mal'cev \cite{Mal49}, any nilmanifold $\Nil$ can be equipped with a Mal'cev basis adapted to the lower central series filtration. By (iii) of the above definition, given a Mal'cev basis $\mathcal{X}=\{X_1, \ldots, X_m\}$, each $g \in G$ can be uniquely expressed as
\begin{equation*}
g = \exp(t_1X_1)\exp(t_2X_2)\ldots \exp(t_mX_m).
\end{equation*}
The vector $(t_1,t_2, \ldots, t_m)$ is called the \emph{Mal'cev coordinate} of $g$ and the bijection $\kappa \colon G \rightarrow \mathbb{R}^m$ given as
\begin{equation*}
\kappa (g)= (t_1,t_2, \ldots, t_m)
\end{equation*}
is called the \emph{Mal'cev coordinate map}. Hence $\Gamma = \kappa^{-1}(\mathbb{Z}^m)$.

The Mal'cev basis can be used to define the metric on nilmanifolds. Let $\Nil$ be a $m$-dimensional nilmanifold with a Mal'cev basis $\mathcal{X}$. The corresponding Mal'cev coordinate map is denoted by $\kappa$. Then the metric on $G$ is defined to be the largest metric $d_G$ such that $d_G(g_1,g_2) \leq |\kappa(g_1 ^{-1}g_2)|$, where $|\cdot|$ denotes the $l^{\infty}$-norm on $\mathbb{R}^m$. This metric can be more explicitly expressed as
\begin{equation*}
d_G(g_1,g_2)=\inf\bigg\{\sum_{i=0}^{n-1}\min(|\kappa(h_{i-1}^{-1}h_i)|,
|\kappa(h_i ^{-1}h_{i-1})|) \colon h_0,\ldots,h_n \in G ; h_0=g_1, h_n = g_2 \bigg\},
\end{equation*}
from which we see that $d_G$ is left-invariant. The above metric on $G$ descends to a metric on $\Nil$ given by
\begin{equation*}
d_{\Nil}(\Gamma g_1, \Gamma g_2) \coloneqq \inf\{d_{G}(g_1'.g_2') \, \colon \, g_1', g_2' \in G, \Gamma g_1=\Gamma g_1', \Gamma g_2 = \Gamma g_2 '\}.
\end{equation*}
It can be proved that $d_{\Nil}$ is indeed a metric on $\Nil$. Since $d_{G}$ is left-invariant, we also have
\begin{equation*}
d_{\Nil}(\Gamma g_1, \Gamma g_2)= \inf_{\gamma \in \Gamma} d_{G}(g_1,\gamma g_2).
\end{equation*}

\Def{Rationality of a Mal'cev basis}{
Let $\Nil$ be a $m$-dimensional nilmanifold and let $Q>0$. A Mal'cev basis $\mathcal{X}=\{X_1,X_2, \ldots ,X_m\}$ for $\Nil$ is called $Q$-\emph{rational} if all of the coefficients $c_{ijk}$ in the relations
\begin{equation*}
[X_i,X_j]=\sum_{k=1}^m c_{ijk}X_k
\end{equation*}
are rational with height $\leq Q$. Here for a rational number $x=a/b$, its \emph{height} is defined to be $\max(|a|,|b|)$.
}
The following lemma is a weak version of \cite[Lemma A.4]{GT12a}.
\begin{LEM}\label{Q-equivalent}
Let $Q \geq 2$ and let $\mathcal{X}$ be a $Q$-rational Mal'cev basis for $\mathfrak{g}$ with the coordinate map $\kappa$. Then for all $g,h \in G$ with $d_G(g,\mathrm{id}_G),d_G(h,\mathrm{id}_G) \leq Q$, we have the bound
\begin{equation}\label{Bound}
|\kappa(g)-\kappa(h)| \leq Q^{O(1)} d_G(g,h),
\end{equation}
where $\mathrm{id}_G$ stands for the identity element of $G$.
\end{LEM}

\section{Appendix II: the distality of $(\TNil,T)$}

The purpose of this section is to establish the following theorem that implies the distality of the flow $(\TNil,T)$.

\begin{THM}\label{Thm/Distal}
Let $\mathbb{T}$ be the unit circle and $\Nil$ the $3$-dimensional Heisenberg nilmanifold.
Let $\alpha \in [0,1)$ and let $\varphi_1, \varphi_2, \psi$ be $C^{\infty}$-smooth periodic functions with period $1$.
Denote by $S$ the skew product
\begin{equation}\label{8.SForm}
S: (t,\Gamma g) \mapsto \lb t+\alpha, \Gamma g \bmat{\varphi_1(t)}{\varphi_2(t)}{\psi(t)}\rb.
\end{equation}
Then the flow $(\TNil, S)$ is distal.
\end{THM}

\begin{proof}
Recall that the metric on $\TNil$ is given by (\ref{6.DefMetric}). Assume on the contrary that $(t_1,\Gamma g_1) \neq (t_2, \Gamma g_2) \in \TNil$ but
\begin{equation*}
\lim_{k \rightarrow \infty }d(S^{n_k}(t_1,\Gamma g_1),S^{n_k}(t_2 ,\Gamma g_2)) = 0
\end{equation*}
for some sequence $n_k \to \infty$. Then we must have $t_1=t_2$ since $S$ performs as a rotation on the first component, which preserves the metric on $\mathbb{T}^1$. Therefore, the distance of the second components of $S^{n_k}(t_1,\Gamma g_1)$ and $S^{n_k}(t_2 ,\Gamma g_2)$ tends to zero. Since now $t_1=t_2$, by the definition of $S$, we deduce that
there exists a sequence  $\{h_k\}_{k\geq 1}$ in $G$ such that
\begin{equation*}
\lim_{k \rightarrow \infty} d_{\Nil}(\Gamma g_1 h_k, \Gamma g_2 h_k) = 0.
\end{equation*}
In other words, there exist $r_k,s_k \in \Gamma$ such that
\begin{equation}\label{d_Gtend0}
\lim_{k \rightarrow \infty} d_{G}(r_k g_1 h_k, s_k g_2 h_k) = 0.
\end{equation}
Moreover, since $d_{G}$ is left-invariant, we can assume without loss of generality that each component of $r_k g_1 h_k$ lies in $[0,1)$. Therefore, by \eqref{5.MC}, we have
\begin{equation*}
|\kappa(r_k g_1 h_k)| \leq 2
\end{equation*}
and hence
\begin{equation}\label{d_G<2}
d_G(r_k g_1 h_k, \mathrm{id}_G) \leq 2
\end{equation}
by the definition of $d_G$. Hnece by (\ref{d_G<2}) and (\ref{d_Gtend0}), when $k$ is sufficiently large, we have
\begin{equation*}
d_G(s_k g_2 h_k, \mathrm{id}_G) \leq d_G(r_k g_1 h_k, \mathrm{id}_G) + d_{G}(r_k g_1 h_k, s_k g_2 h_k) \leq 4.
\end{equation*}
Applying Lemma \ref{Q-equivalent} with $Q=4$, we obtain that
\begin{equation}\label{kappatend0}
\lim_{k \rightarrow \infty} |\kappa(r_k g_1 h_k)- \kappa(s_k g_2 h_k)| = 0.
\end{equation}
From now on, for $g \in G$, we use $g^1,g^2,g^3$ to denote its $x,y,z$-component, respectively. A simple calculation shows that
\begin{equation*}
r_k g_1 h_k = \bmat{r_k^1 + g_1^1 + h_k^1}{r_k^2 + g_1^2 +h_k ^2}{r_k^3 + g_1^3 +h_k ^3 + g_1^1r_k^2  + h_k^1r_k^2 + g_1^2 h_k ^1}
\end{equation*}
and
\begin{equation*}
s_k g_2 h_k = \bmat{s_k^1 + g_2^1 + h_k^1}{s_k^2 + g_2^2 +h_k ^2}{s_k^3 + g_2^3 +h_k ^3 + g_2^1s_k^2 + h_k^1s_k^2 + g_2^2 h_k ^1}.
\end{equation*}
Therefore, by considering the first component of $\kappa(r_k g_1 h_k)$ and $\kappa(s_k g_2 h_k)$, we deduce from \eqref{kappatend0} that
\begin{equation*}
\lim_{k \rightarrow \infty } (r_k^1 -s_k^1)= g_2^1 -g_1^1.
\end{equation*}
But $r_k^1 - s_k^1 \in \mathbb{Z}$, so for $k$ sufficiently large, $r_k^1 - s_k^1$ is an integral constant $a$ and we have $a=g_2^1 -g_1^1$. Similarly, for $k$ sufficiently large, $r_k^2 - s_k^2$ is an integral constant $b$ satisfying $b = g_2^2 - g_1^2$. Now since for large $k$, the $x,y$-components of $r_k g_1 h_k$ and $s_k g_2 h_k$ are equal, by \eqref{kappatend0} and the definition of $\kappa$, the difference between their $z$-components tends to zero as well. So for $k$ sufficiently large we have
\begin{eqnarray*}
&&(r_k^3 + g_1^3 +h_k ^3 + g_1^1r_k^2  + h_k^1r_k^2 + g_1^2 h_k ^1)
-(s_k^3 + g_2^3 +h_k ^3 + g_2^1s_k^2 + h_k^1s_k^2 + g_2^2 h_k ^1) \\
&&=(r_k^3-s_k^3)+(g_1^3 - g_2^3)+g_1^1r_k^2 - g_2^1 s_k^2 \\
&&=(r_k^3-s_k^3)+(g_1^3 - g_2^3)+g_1^1r_k^2 - (g_1^1 + a) s_k^2 \\
&&=(r_k^3-s_k^3)+(g_1^3 - g_2^3)+g_1^1b -  a s_k^2
\end{eqnarray*}
which approaches $0$ as $k\to\infty$.
Again, since $r_k^3-s_k^3-a s_k^2 \in \mathbb{Z}$, there exists an integral constant $c$ such that $c=r_k^3-s_k^3-a s_k^2$ for large $k$ and $c$ satisfies $c= g_2^3 - g_1^3 - bg_1^1$. As a consequence, we have found $a,b,c \in \mathbb{Z}$ such that
\begin{equation*}
\begin{cases}
g_2^1 = g_1^1 +a, \\
g_2^2 = g_1^2 +b, \\
g_2^3 = g_1^3 + bg_1^1 + c,
\end{cases}
\end{equation*}
which implies
\begin{equation*}
\Gamma g_2=\bmat{g^1_2}{g^2_2}{g^3_2}= \Gamma \bmat{a}{b}{c}\bmat{g^1_1}{g^2_1}{g^3_1} =\Gamma g_1.
\end{equation*}
This is a contradiction, and the theorem is proved.
\end{proof}

\medskip
\noindent
{\bf Acknowledgements.}
{The first author  is partially supported by the National
Science Foundation of China under Grants 11431012, 11731003.} The second and third authors are partially supported by the National
Science Foundation of China under Grant 11531008, the Ministry of Education of China
under Grant IRT16R43, and the Taishan Scholar Project of Shandong Province.

\end{document}